\documentclass[11pt,a4paper, leqno]{amsart}
\usepackage[top = 3cm, bottom=3cm, left=3cm,right=3cm]{geometry}
\usepackage[utf8]{inputenc}
\usepackage{amsmath}
\usepackage{longtable}
\usepackage{amsfonts}
\usepackage{diagbox}
\usepackage{mathtools}
\usepackage{amssymb}
\usepackage{booktabs}
\usepackage{stmaryrd}
\usepackage{comment}
\usepackage{multirow}

\usepackage{amsmath}
\usepackage{amssymb}
\usepackage{amsthm}
\usepackage{latexsym}
\usepackage{mathtools}
\mathtoolsset{showonlyrefs,showmanualtags}
\usepackage{thmtools}
\usepackage{amsfonts}
\usepackage{mathrsfs}
\usepackage{textcomp}
\usepackage{graphicx}
\usepackage{setspace}
\usepackage{nicefrac}
\usepackage{indentfirst}
\usepackage{enumerate}
\usepackage{wasysym}
\usepackage[normalem]{ulem}
\usepackage{upgreek}

\usepackage{paralist}
\usepackage{xcolor}
\usepackage{etoolbox}
\usepackage{mathdots}
\usepackage{wrapfig}
\usepackage{floatflt}
\usepackage{tensor} 
\usepackage{parskip}
\usepackage{lscape}
\usepackage{stmaryrd}
\usepackage[enableskew]{youngtab}
\usepackage{ytableau}
\usepackage{pifont}

\usepackage{tikz}
\usetikzlibrary{arrows}
\usetikzlibrary{matrix}
\usetikzlibrary{positioning}
\usetikzlibrary{snakes}
\usetikzlibrary{calc}

\usepackage[all,cmtip]{xy}
\usepackage[labelfont=bf, font={small}]{caption}[2005/07/16]

\usepackage{xcolor}
\definecolor{red}{rgb}{1,0,0}

\newcommand{\vvirg}{ , \dots , }

\newcommand{\calF}{\mathcal{F}}

\newcommand{\calH}{\mathcal{H}}

\newcommand{\calQ}{\mathcal{Q}}

\newcommand{\calS}{\mathcal{S}}

\newcommand{\scrC}{\mathscr{C}}

\newcommand{\scrF}{\mathscr{F}}

\newcommand{\bbC}{\mathbb{C}}

\newcommand{\bbP}{\mathbb{P}}

\newcommand{\bbX}{\mathbb{X}}

\newcommand{\bbZ}{\mathbb{Z}}

\newcommand{\frakm}{\mathfrak{m}}

\renewcommand{\phi}{\varphi}
\newcommand{\eps}{\varepsilon}

\newcommand{\dashto}{\dashrightarrow}

\renewcommand{\tilde}[1]{\widetilde{#1}}

\renewcommand{\bar}[1]{\overline{#1}}

\newcommand{\rank}{\mathrm{rank}}

\DeclareMathOperator{\codim}{codim}

\DeclareMathOperator{\Sym}{Sym}

\newcommand{\sat}{\mathit{sat}}

\newcommand{\GL}{\mathrm{GL}}
\newcommand{\SL}{\mathrm{SL}}
\newcommand{\PGL}{\mathrm{PGL}}

\newcommand{\Macaulay}{\texttt{Macaulay2}}

\usepackage[textwidth=0.9in]{todonotes}

\newcommand{\PP}{\mathbb{P}}

\usepackage[breaklinks,bookmarks=false, hidelinks]{hyperref} 
\hypersetup{
    colorlinks,
    linkcolor={red!50!black},
    citecolor={green!50!black},
    urlcolor={blue!80!black}
}%

\newtheorem{theorem}{Theorem}[section]
\newtheorem{lemma}[theorem]{Lemma}
\newtheorem{proposition}[theorem]{Proposition}
\newtheorem{corollary}[theorem]{Corollary}
\theoremstyle{definition}

\newtheorem*{question*}{Question}

\newtheorem{example}[theorem]{Example}

\newcommand{\Str}{\mathrm{Str}}

\newcommand{\C}{\mathbb{C}}
\newcommand{\CC}{\mathbb{C}}
\newcommand{\cO}{\mathcal{O}}

\DeclareMathOperator{\Gr}{Gr}
\DeclareMathOperator{\HH}{H}

\title{Collineation varieties of tensors}
\author{Fulvio Gesmundo and Hanieh Keneshlou}

\address[F. Gesmundo]{Institut de Mathématiques de Toulouse; UMR5219 -- Université de Toulouse; CNRS -- UPS, F-31062 Toulouse Cedex 9, France}
\email{fgesmund@math.univ-toulouse.fr}

\address[H. Keneshlou]{Fachbereich Mathematik und Statistik, Universität Konstanz, Germany}
\email{hanieh.keneshlou@uni-konstanz.de}


\begin{document}

\begin{abstract}
In this article, we introduce the $k$-th collineation variety of a third order tensor. This is the closure of the image of the rational map of size $k$ minors of a matrix of linear forms associated to the tensor. We classify such varieties in the case of pencils of matrices, and nets of matrices of small size. We discuss the natural stratification of tensor spaces induced by the invariants and the geometric type of the collineation varieties. 
\end{abstract}

\maketitle
\section{Introduction}

A fundamental problem in mathematics is to determine conditions to distinguish different objects up to some natural equivalence relation. Of particular interest in several areas of geometry and applied sciences is the study of orbits and orbit-closures of linear spaces under a natural group action, varieties of tensors arising from such actions, and more generally varieties of tensors sharing interesting properties. The classical study of secant varieties \cite{Sylv:PrinciplesCalculusForms,Terr:seganti,LanMan:IdealSecantVarsSegre,LanWey:IdealsSingularitiesSecantVarietiesSegreVarieties}, the classification of small orbits for natural group actions \cite{Vinberg:ClassificationHomogNilpElements,LanMan:ProjectiveGeometryRatHomVars,ManMic:SecantsMinuscule}, classification results for finite-dimensional algebras \cite{BlaLys:DegenTensorsAlgebras,BES} all fit in this line of research. Important applications include the resource theory of entanglement in quantum physics \cite{BernCaru:AlgGeomToolsEntanglement,walter2016multipartite}, the characterization of complexity measures in theoretical computer science \cite{BuClSho:Alg_compl_theory,Lan:GeometryComplThBook}, the study of certain classes of probability models in algebraic statistics \cite{Lauritzen:GraphicalModels,AmFauSt:MomentVarsGaussian} and many others. 

A classical method to determine equations of varieties of tensors is by studying invariants of associated varieties. Here, with the term \emph{equation}, we simply mean a ``Zariski closed condition that is (reasonably) easy to check''. For instance, classical equations for varieties of tensors include rank conditions on their flattenings, and more generally \emph{unexpected} intersection properties of the image of a flattening with special subvarieties. In this work, we focus on tensors of order three and we study their \emph{collineation varieties}. We naturally identify a tensor of order three with a linear space of matrices, a classical point of view in algebraic geometry \cite{SatKim:ClassificationIrredPrehomVS,Man:PrehomogeneousSpaces,HuaLan:LinSpacesMatricesBounded}. The $k$-th collineation variety of a tensor is the closure of the image of the rational polynomial map sending a matrix of the corresponding linear space to its tuple of $k \times k$ minors. This is a restricted version of the variety of complete collineations, a bilinear analog of the variety of complete quadrics \cite{Tyr:CompleteQuadrics,Vain:CompleteCollineations}. We will provide a more precise description in \autoref{sec: preliminaries}. 

One important fact is that dimension and degree of the $k$-th collineation variety are lower semicontinuous invariant: more precisely, for every $n$, the set of tensors whose $k$-th collineation variety has dimension at most $n$ is closed; moreover, for every $\delta$, the set of tensors whose $k$-th collineation variety has dimension (exactly) $n$ and degree bounded by $\delta$ is a Zariski closed set in the locally closed set of tensors whose $k$-th collineation variety has dimension $n$. The study of the degrees of collineation varieties was proposed in \cite{ConMic:CharNumChromNumTensors} as a method to provide new equations on the space of tensors. A more general version of these degrees, called \emph{characteristic numbers} of the tensor in \cite{ConMic:CharNumChromNumTensors}, unifies several invariants from different areas such as chromatic numbers of graphs, Betti numbers of (complements of) hyperplane arrangements, characteristic numbers of quadrics in Schubert calculus and many others.

In this work, we propose the study of finer invariants of the collineation varieties, beyond their degrees. In fact, an important open question concerns \emph{how well} the $k$-th collineation variety can distinguish between different orbit-closures of tensors, under the action of the product of general linear groups acting on the tensor factors. We answer this problem in a restricted range.
\begin{itemize}
    \item For tensors in $\bbC^2 \otimes \bbC^m \otimes \bbC^n$, that is pencils of matrices, we prove that the collineation varieties are always rational normal curves, and their degree is controlled by the intersection of the pencil with the variety of matrices of bounded rank. See \autoref{thm: pencils general}.
    \item For tensors in $\bbC^3 \otimes \bbC^m \otimes \bbC^n$ with $m=2$ or $(m,n) = (3,3)$, we give a full classification of all possible collineation varieties. See \autoref{thm: small nets}, \autoref{thm: nets smooth}, and \autoref{thm: netC3}. We observe that they are always varieties of minimal degree. 
    \item We introduce a stratification of the space of tensors based on the collineation varieties. We realize some known interesting varieties of tensors, related to the notion of tensor rank and tensor subrank, in terms of these strata. See \autoref{thm: sigma3 as F1}, and \autoref{prop: subrank components}.
\end{itemize}
The matrix pencils and the small nets that we consider are the only spaces of tensors of order three with finite or tame orbit structure, in the sense of Gabriel's Theorem \cite{Gabriel-quivers,Kac:InfiniteRootSystemsRepGraphsInvariantTheory}: in particular, there are strong classification results providing \emph{normal forms} for tensors in this range. This gives us access to the entire space of tensors with only a small number of parameters and allows us to provide a full classification of the collineation varieties of all tensors.

A similar characterization, limited to the study of the degree, was provided in \cite{FMS} for the case of pencils of symmetric matrices, namely partially symmetric tensors in $\bbC^2 \otimes \bbC^n \otimes \bbC^n$. Moreover, applications in the theory of maximum likelihood motivated the study of the \emph{top} collineation variety of partially symmetric tensors, that is the image of the (restriction to a linear space of the) map sending a matrix to its inverse \cite{JiKoWi:LinSpacesSymmetricMatrices,DyKoRySi:MaxLikelihoodNetsConics}.

\subsection*{Acknowledgement}
HK is funded by the Deutsche Forschungsgemeinschaft — Projektnummer 467575307. We thank Jarek Buczy\'nski and Mateusz Micha{\l}ek for helpful discussions on the topics of this work.

\section{Preliminaries}\label{sec: preliminaries}

Let $V_1,V_2,V_3$ be complex finite-dimensional vector spaces, with $n_i = \dim V_i$. We study tensors in $V_1 \otimes V_2 \otimes V_3$ up to the natural action of the product of general linear groups $\GL(V_1) \times \GL(V_2) \times \GL(V_3)$.

A tensor $T \in V_1 \otimes V_2 \otimes V_3$ defines three natural linear maps 
\[
V_1^* \to V_2 \otimes V_3 \qquad  V_2^* \to V_1 \otimes V_3 \qquad  V_3^* \to V_1 \otimes V_2 ,
\]
via tensor contraction: these are the flattenings of $T$. We say that a tensor is concise if the three flattening maps are injective. Notice that, up to the action of $\GL(V_1)$, a tensor $T \in V_1 \otimes V_2 \otimes V_3$ is uniquely determined by the image $T(V_1^*)$ of its first flattening: in other words, if $T,T'$ are two tensors such that the $T(V_1^*) = T'(V_1^*)$ as linear subspaces of $V_2 \otimes V_3$, then there exists $g_1 \in \GL(V_1)$ such that $g_1 T = T'$. A similar statement is valid on the other factors.

The image of the first flattening map is a linear subspace of $V_2 \otimes V_3$; in coordinates, this can be regarded as a $n_2 \times n_3$ matrix whose entries are linear forms in (at most) $n_1$ variables. The $(k \times k)$-minors of this matrix are a collection of $\binom{n_2}{k} \binom{n_3}{k}$ polynomials in (at most) $n_1$-variables. The \emph{$k$-th collineation variety} of $T$ (on its first factor), denoted by $\scrC^1_k(T)$, is the variety in $\bbP (\bbC^{\binom{n_2}{k} \binom{n_3}{k}})$ parameterized by this collection of polynomials. Similarly, one can define collineation varieties on the second factor $\scrC^2_k(T)$ or the third factor $\scrC^3_k(T)$, considering the other flattenings of the tensor.

The definition can be given invariantly as follows. Every linear map $f: V \to W$ induces a map on symmetric powers $f^{\cdot k} : S^k V \to S^k W$. Moreover, a linear map $F : S^d V \to U$ can be regarded as a polynomial map $F: V \to U$ defined by homogeneous polynomials of degree $d$. The symmetric power $S^k (V_2 \otimes V_3)$ contains a canonical copy of the $(\GL(V_2) \times \GL(V_3))$-representation $\Lambda^k V_2 \otimes \Lambda^k V_3$. Therefore, there is a canonical equivariant projection $S^k (V_2 \otimes V_3) \to \Lambda^k V_2 \otimes \Lambda^k V_3$. 

Given a tensor $T \in V_1 \otimes V_2 \otimes V_3$, let $\mu^{1}_{k,T}$ be the composition of the $k$-th symmetric power of the flattening map $T: V_1^* \to V_2 \otimes V_3$ followed by the equivariant projection onto $\Lambda^k V_2 \otimes \Lambda^k V_3$, regarded as a polynomial map $\mu^{1}_{k,T}: V_1^* \to \Lambda^k V_2 \otimes \Lambda^k V_3$. Its projectivization gives a rational map
\[
\mu^{1}_{k,T} :\PP(V_1^*) \dashto \PP(\Lambda^k V_2 \otimes \Lambda^k V_3).
\]
The closure of the image of $\mu^{1}_{k,T}$ defines the subvariety  $\scrC^1_{k}(T)\subset \bbP (\Lambda^k V_2 \otimes \Lambda^k V_3)$. Similarly, one can define analogous varieties corresponding to the other factors. We point out that $\scrC^1_{k}(T)$ is well defined if the image $T(V_1^*)$ of the first flattening of $T$ contains matrices of rank at least $k$. In fact, consider the classical determinantal variety
\[
 \sigma_k^{V_2 \otimes V_3} = \{ A \in \bbP (V_2 \otimes V_3)  : \rank(A) \leq k\} \subseteq \bbP (V_2 \otimes V_3)
\]
where $\rank(-)$ denotes the rank as a linear map $V_2 ^* \to V_3$. Then $\scrC^1_{k}(T)$ is well-defined if $\PP(T(V_1^*)) \not\subseteq \sigma_{k-1}^{V_2 \otimes V_3}$. The \textit{$k$-th base locus scheme} of $T$ is
\[
B^1_k(T):=\bbP (T(V_1^*)) \cap \sigma_{k-1}^{V_2 \otimes V_3}
\]
and it contributes directly to the geometry of $\scrC^1_k(T)$. We provide some details on this in \autoref{subsec: linear series and projections}

If $n_2 = n_3 = k+1$, then $ \scrC^1_{k}(T)$ is the so-called \emph{reciprocal variety} studied in \cite{DyKoRySi:MaxLikelihoodNetsConics,JiKoWi:LinSpacesSymmetricMatrices} in the symmetric setting. An alternative definition of the collineation variety can be given via projection from the \emph{variety of complete collineations}. We refer to \cite{ConMic:CharNumChromNumTensors,MMMSV:CompleteQuadrics} for details on this point of view. 

In the following, we will always assume our tensor of study is concise. This is not restrictive: if a tensor $T \in V_1 \otimes V_2 \otimes V_3$ is not concise, then there exist subspaces $V_i' \subseteq V_i$ with at least one strict inclusion such that $T \in V_1' \otimes V_2' \otimes V_3'$. If a tensor is concise, then the flattening map $T : V_1^* \to V_2 \otimes V_3$ is injective and defines an isomorphism between $T(V_1^*)$ and $V_1^*$. In the following, we will identify these two spaces. 

\subsection{Projections of Veronese varieties}\label{subsec: linear series and projections}

We provide a geometric description of the $k$-th collineation variety of a tensor $T \in V_1 \otimes V_2 \otimes V_3$ as a projection of the $k$-th Veronese embedding of $\bbP V_1^*$.

We consider the following general framework. Let $L \subseteq S^k V_1$ be a linear space of homogeneous polynomials of degree $k$. Let $B \subseteq \bbP V_1^*$ be its base locus and let $I(B) \subseteq \bbC[V_1^*] \simeq \Sym (V_1^*)$ be the ideal of $B$. We have $L \subseteq I(B)_k$. In fact, $L$ cuts out $B$ scheme-theoretically, in the sense that $I(B)$ coincides with the saturation of the ideal generated by $L$; as a consequence $I(B)_k$ cuts out $B$ scheme-theoretically as well.

From a dual point of view, the projectivization $\bbP I(B)^\perp_k \subseteq \bbP S^k V_1^*$ coincides with the space $\langle \nu_k(B) \rangle$ \cite[Proposition 4.4.1.1]{Lan:TensorBook}; moreover, $\langle \nu_k(B) \rangle = \bbP I_k(B)^\perp \subseteq \bbP L^\perp$. Since $L$ cuts out $B$ scheme-theoretically, one has $\bbP L^\perp \cap \nu_k(\bbP V_1^*) =\nu_k( B)$: this follows from an argument similar to \cite[Corollary 2.5]{GesKayTel:ChoppedIdeals}.

Now, the linear space $L$ defines a rational map $\phi_L : \bbP V_1^* \dashto \bbP L^*$ by sending a point $p$ to the hyperplane in $L$ defined by $\{ f \in L : f(p) = 0\}$. In coordinates, if $f_0 \vvirg f_m$ is a basis of $L$, this can be interpreted as the evaluation $\phi_L(p) = (f_0(p): \cdots : f_m(p)) \in \bbP^m$. 

The natural identification $L^* \simeq S^k V_1^* / L^\perp$ gives the following commutative diagram,
\begin{equation}\label{eqn: diagram veronese projection}
\xymatrixcolsep{3pc}
\xymatrix{
& \bbP S^k V_1^* \ar@{-->}[d]^{\pi} \\
    \bbP V_1^*  \ar[ur]^{\nu_k} \ar@{-->}[r]_{\phi_L} & \bbP L^*
}
\end{equation}
where $\pi : \bbP S^k V_1^* \dashto \bbP L^*$ is the projection induced by the linear projection $S^k V_1^* \mapsto S^k V_1^* / L^\perp$. The image $X_L = \bar{\phi_L ( \bbP V_1^*)}$ is therefore the projection of $\nu_k(\bbP V_1^*)$ from the linear space $\bbP L^\perp$; if $L = I(B)_k$, this is the same as the projection from $\langle \nu_k(B) \rangle$. 

We may apply this construction to the setting of the $k$-th collineation variety. In this case, $L$ is the linear span of the size $k$ minors of $T(V_1^*)$, regarded as polynomials on $\bbP V_1^*$. Therefore, $\scrC^1_k(T)$ is the projection of $\nu_k(\bbP V_1^*)$ and if $L$ coincides with the degree $k$ homogeneous component of $I(B_k^1(T))$, then the center of the projection coincides with $\langle \nu_k( B_k^1(T)) \rangle$. We can enhance the diagram in \eqref{eqn: diagram veronese projection} as follows
\[
\xymatrixcolsep{3pc}\xymatrix{
& \bbP S^k V_1^* \ar@{-->}[d]^{\pi} \ar[r]^{T^{\cdot k}} & \bbP S^k (V_2 \otimes V_3) \ar@{-->}[dd]^{\pi_{D_k}} \\
    \bbP V_1^*  \ar[dr]_{T} \ar[ur]^{\nu_k} \ar@{-->}[r]_{\phi_L} & \bbP L^* \ar[dr]^{\iota}& \\
                   & \bbP(V_2 \otimes V_3) \ar@{-->}[r]_{\phi_{D_k}} \ar[uur]
|!{[u];[r]}\hole^{\nu_k} & \bbP (\Lambda^k V_2 \otimes \Lambda^k V_3)
}
\]
Here $D_k$ denotes the degree $k$ component of the ideal of size $k$ minors on $V_2 \otimes V_3$, which is a subspace of $S^k (V_2 \otimes V_3)^*$. The map $\pi_{D_k}$ is the linear projection with center $D_k^\perp \subseteq S^k (V_1 \otimes V_2)$. The map $\iota$ is the linear embedding induced by $\pi_{D_k} \circ T^{\cdot k}$. The map $\mu^{1}_{k,T}$ then coincides with any of the compositions from $\bbP V_1^*$ to $\bbP (\Lambda^k V_2 \otimes \Lambda^k V_3)$.

\section{Pencils}\label{sec: pencils}
We note that if $\dim V_1 = 1$, a well-defined collineation variety is a single point. In this section, we start with the first interesting case of study, the case of pencils of matrices, that is $\dim V_1 = 2$.

The main result of this section is the following:
\begin{theorem}\label{thm: pencils general}
Let $T \in V_1 \otimes V_2 \otimes V_3$ be a tensor. Let $n_i = \dim V_i$, with $n_1 = 2$ and $n_2 \leq n_3$ and let $1\leq k\leq n_2$ with the second inequality strict if $n_2=n_3$. Suppose $B_k^1(T)$ is $0$-dimensional and $p=\deg(B_k^1(T))$. If $p < k$, then $\scrC^1_k(T)$ is the rational normal curve of degree $k-p$; if $p = k$ then $\scrC^1_k(T)$ is a point. 
\end{theorem}
We point out that if the assumption that $B_k^1(T)$ is $0$-dimensional in \autoref{thm: pencils general} is not satisfied, then $B_k^1(T) = \bbP T(V_1^*)$ and the collineation variety is not even defined. The proof of \autoref{thm: pencils general} is built on the technical \autoref{lemma: binary determinantal are saturated} whose proof uses Kronecker's classification of matrix pencils \cite[Ch. XIII]{Gant:TheoryOfMatrices}. Indeed, matrix pencils, or equivalently tensors in $V_1 \otimes V_2 \otimes V_3$ with $\dim V_1 = 2$, have \emph{normal forms} with respect to the action of $\GL(V_1) \times \GL(V_2) \times \GL(V_3)$. Up to the action of $\GL(V_1)$, a tensor is uniquely determined by  the linear space $T(V_1^*)$, which may be regarded as an $n_2 \times n_3$ matrix whose entries are binary forms in two variables $x_0,x_1$. The Kronecker classification of matrix pencils guarantees that $T$ can be normalized so that $T(V_1^*)$ is represented by a block diagonal matrix, and the diagonal blocks fall into three classes of \emph{irreducible pencils}:
\begin{align*}
L_h &=  \left( \begin{array}{ccccc}
             x_0 & x_1 & & & \\
              & x_0 & x_1 & & \\
             & & \ddots & \ddots &  \\
             & & &  x_0 & x_1
            \end{array}
\right)   \in \bbC^2 \otimes \bbC^h \otimes \bbC^{h+1}; \\ 
R_h &=   \left( \begin{array}{cccc}
             x_0 &  &  \\
              x_1 & x_0 &   & \\
                & x_1 & \ddots  & \\
              & & \ddots & x_0  \\
              & &   & x_1
            \end{array}
\right)   \in \bbC^2 \otimes \bbC^{h+1} \otimes \bbC^{h}; \\
J_h(\lambda) &=   \left( \begin{array}{cccc}
             x_0  + \lambda x_1 & x_1 &  \\
               & \ddots & \ddots  & \\
              & & \ddots & x_1  \\
              & &   & x_0 +\lambda x_1
            \end{array}
\right)   \in \bbC^2 \otimes \bbC^h \otimes \bbC^{h}, \lambda\in \bbC .
\end{align*}
The pencils $L_h,R_h$ are called, respectively, the left and right singular pencils. The pencil $J_h(\lambda)$ is called the Jordan pencil with eigenvalue $\lambda$. Note that for $h = 1$, $J_1(\lambda) = (x_0 + \lambda x_1)$ is a linear form.

\autoref{lemma: binary determinantal are saturated} shows that the $r$-th determinantal ideal of a matrix pencil coincides with its saturation in degree $r$. 
\begin{lemma}\label{lemma: binary determinantal are saturated}
 Let $M=M(x_0,x_1)$ be an $n_2 \times n_3$ matrix of linear forms with entries in $\bbC[x_0,x_1]$. Let $I(M,r) \subseteq \bbC[x_0,x_1]$ be the ideal generated by the $r \times r$ minors of $M$. Then
 \[
  I(M,r)^{\sat}_r = I(M,r)_r.
 \]
\end{lemma}
\begin{proof}
By Kronecker's classification, $M(x_0,x_1)$ can be reduced in block form whose blocks are irreducible pencils. The proof is by induction on the number of blocks. We first prove the statement for the irreducible pencils. Let $r \leq h$. For the matrix pencils $M=L_h$ and $M=R_h$ the statement is clear: $(I(M,r))_r$ contains all binary forms of degree $r$. In particular, $I(M,r)^{\sat} = (1)$ and $(I(M,r))_r  = (I(M,r)^{\sat})_r$ as desired. For the Jordan pencil $M=J_h(\lambda)$, we consider two cases. If $r < h$, then it is easy to verify that $I(M,r)_r$ contains all binary forms of degree $r$: as before, $I(M,r)^{\sat} = \bbC[x_0,x_1]$ and $(I(M,r))_r  = (I(M,r)^{\sat})_r$ as desired. If $r= h$, then $I(M,r) = (x_0 + \lambda x_1)^r$ coincides with its saturation. 

If the pencil is not irreducible, then it is a sum of irreducible pencils, that is
\[
M(x_0,x_1) = \left(
\begin{array}{ccc}  
M_1(x_0,x_1) & & \\
& \ddots & \\
&  & M_s(x_0,x_1)\\
\end{array}
\right),
\]
where $M_j(x_0,x_1)$ is an irreducible pencil. In this case, write $M(x_0,x_1) = M_1 (x_0,x_1) \boxplus \cdots \boxplus M_s(x_0,x_1)$. Order the blocks so that the Jordan blocks appear last; as a consequence, if $M_s(x_0,x_1)$ is not a Jordan block, then there are no Jordan blocks in the representation of $M$. 

We use induction on the number of blocks $s$ to prove a slightly stronger statement. We show that $I(M,r)_r = I(M,r)^\sat_r$ and if there are no Jordan blocks in the representation, then either $I(M,r)_r = \bbC[x_0,x_1]_r$ or $I(M,r)_r = 0$. The base of the induction for $s = 1$ follows from the discussion on the irreducible pencils. 

If $s \geq 2$, let $N = M_1 \boxplus \cdots \boxplus M_{s-1}$, so that $M = N \boxplus M_s$. By the induction hypothesis, the statement holds for $N$. In particular, for every $p = 0 \vvirg r$, $(I(N,p))_p = (I(N,p)^\sat) _ p = f_p \cdot \bbC[x_0,x_1]_{p-d_p}$, where $f_p \in \bbC[x_0,x_1]$ is a binary form of degree $d_p$. By construction, $f_p$ is a divisor of $f_{p+1}$; moreover $f_p = 0$ if $I(N,p) = 0$ and $f_p = 1$ if $I(N,p) = \bbC[x_0,x_1]$.

First, suppose $M_s = L_h$ or $M_s = R_h$ from some $h$. In this case, there are no Jordan blocks in the representation. By the induction hypothesis either $I(N,p)_p = 0$, namely $f_p = 0$, or $I(N,p)_p = \bbC[x_0,x_1]_p$, namely $f_p = 1$. The block diagonal structure of $M$ guarantees
\[
I(M,r)_r = \sum_{p=0}^r  I(N,r-p)_{r-p} \cdot I(M_s,p)_{p} .
\]
If there is $p \leq r$ such that $I(N,r-p)_{r-p} \neq 0$, then $I(M,r)_r = \bbC[x_0,x_1]_{r-p} \cdot \bbC[x_0,x_1]_{p} = \bbC[x_0,x_1]_r$, otherwise $I(M,r) = 0$. This proves the desired assertion.

If $M_s = J_h(\lambda)$ is a Jordan block for some $h$ and $\lambda$, then, as before 
\[
I(M,r)_r = \sum_{q=0}^r  I(N,q)_{q} \cdot I(M_s,r-q)_{r-q} .
\]
Let $\bar{q} = \max\{ q \leq r: I(N,q)_q \neq 0\}$ and let $\bar{p} = r-\bar{q}$. 

If $\bar{p} > h$, then there are no nonzero terms in the summation and $I(M,r)_r = 0$.

If $\bar{p} \leq h$, then the relevant range of the summation is $q =r-h \vvirg \bar{q}$, and correspondigly $p = r-q = h \vvirg \bar{p}$ (in descending order). The summand corresponding to $q = r-h$ is 
\[
f_{r-h}\cdot \bbC[x_0,x_1]_{r-h -d_{r-h}} \cdot (x_0+\lambda x_1)^h = (x_0+\lambda x_1)^h f_{r-h} \cdot \bbC[x_0,x_1]_{r-h -d_{r-h}}.
\]
The summands for $q > r-h$ are  
\[
f_{q}\cdot \bbC[x_0,x_1]_{q -d_{q}} \cdot \bbC[x_0,x_1]_{r-q} = f_q \cdot \bbC[x_0,x_1]_{r-d_{q}} .
\]
By construction, for every $q$ with $r-h < q \leq \bar{q}$, we have that $f_q$ is divisible by $f_{r-h+1}$. Therefore, 
\[
f_q \cdot \bbC[x_0,x_1]_{r-d_{q}} \subseteq f_{r-h + 1} \cdot \bbC[x_0,x_1]_{r-h+1 - d_{r-h+1}}
\]
and we obtain
\[
I(M,r)_r = (x_0+\lambda x_1)^h f_{r-h} \cdot \bbC[x_0,x_1]_{r-h -d_{r-h}} + f_{r-h + 1} \cdot \bbC[x_0,x_1]_{r-h+1 - d_{r-h+1}}.
\]
Let $g = \gcd(f_{r-h + 1},(x_0+\lambda x_1)^h f_{r-h})$. Since $f_{r-h+1}$ is a multiple of $f_{r-h}$, we have $g = (x_0+\lambda x_1)^{h'}f_{r-h}$ for some $h' \leq h$; here $h'$ is the largest power of $(x_0+\lambda x_1)$ dividing $f_{r-h+1}/f_{r-h}$. In particular $\deg (g ) = d_{r-h} + h'$.

We prove that $I(M,r)_r = g \cdot \bbC[x_0,x_1]_{r-\deg(g)}$. To see this, it suffices to show that the relatively prime polynomials $g_1 = \frac{1}{g} f_{r-h + 1}$ and $g_2 = \frac{1}{g} (x_0+\lambda x_1)^h f_{r-h}$ generate $\bbC[x_0,x_1]_{r-\deg(g)}$. It is a general fact that if $g_1,g_2$ are coprime of degree $\delta_1,\delta_2$ respectively, then $(g_1,g_2)_{\rho} = \bbC[x_0,x_1]_{\rho}$ if $\delta_1+ \delta_2 - 1 \leq \rho$. In our case
\begin{align*}
\delta_1 + \delta_2 -1 &= d_{r-h+1} - \deg(g) + d_{r-h} - \deg(g) - 1  \\
&= d_{r-h+1} + h - h' - 1 - (d_{r-h} + h') \leq r- \deg(g).
\end{align*}
This concludes the proof.
\end{proof}

\autoref{thm: pencils general} is a straightforward consequence of \autoref{lemma: binary determinantal are saturated}.
\begin{proof}[Proof of \autoref{thm: pencils general}]
Let $M(x_0,x_1) = T(V_1^*)$ be the image of the first flattening, regarded as a matrix of linear binary forms. Let $J$ be the ideal generated by the minors of size $k$ of $M(x_0,x_1)$. By \autoref{lemma: binary determinantal are saturated}, we have $J_k = J_k^{\sat}$. From the assumptions, $B_{k}^1(T)$ is a $0$-dimensional scheme of degree $p$, which, by the conciseness of $T$, can be regarded as a subscheme of $\bbP (V_1^*)$. In particular, $I(B_1^k(T)) = J^\sat = (f)$ for a homogeneous polynomial $f$ of degree $p$. Write $\frakm = (x_0,x_1)$ for the irrelevant ideal of $\bbC[x_0,x_1]$; since $J_k = J^\sat_k$, we have $J_k = (f \cdot \frakm^{k-p})_k$.

This shows that $\mu^1_{k,T}$ extends to a morphism $\mu^{1}_{k,T} : \bbP V_1^* \to \bbP (\Lambda^k V_2 \otimes \Lambda^k V_3)$ defined by $(x_0,x_1) \mapsto (x_0^{k-p}, x_0^{k-p-1}x_1 \vvirg x_1^{k-p})$, up to a linear transformation of $\bbP (\Lambda^k V_2 \otimes \Lambda^k V_3)$. This shows that $\scrC^1_k(T)$ is parameterized by all monomials of degree $k-p$, hence it is the rational curve of degree $k-p$.
\end{proof}

We record an easy consequence of \autoref{thm: pencils general} which characterizes the least and the most degenerate situations, that is when the collineation varieties are rational normal curves of maximal possible degree and when they are linear spaces.
\begin{corollary}
Let $T \in V_1 \otimes V_2 \otimes V_3$ be a concise tensor with $n_i = \dim V_i$, such that $n_1 = 2$ and $n_2\leq n_3$. 
\begin{enumerate}[(i)]
    \item Suppose $n_1 = n_2$. Then $B^1_{n_1-1}(T) = \emptyset$ if and only if $\scrC^1_{n_1-1}(T)$ is a rational normal curve of degree $n_1-1$.
    \item Suppose $n_1 < n_2$. Then $B^1_{n_1}(T) = \emptyset$ if and only if $\scrC^1_{n_1}(T)$ is a rational normal curve of degree $n_1$.
    \item $B^1_2(T) \neq \emptyset$ if and only if all (well-defined) collineation varieties are either $\bbP^1$ or a single point. 
    \end{enumerate}
\end{corollary}
\begin{proof}
The proof of (i) and (ii) is an immediate consequence of \autoref{thm: pencils general}.

For part (iii), let $t$ be the largest integer such that $\bbP (T(V_1^*)) \not \subseteq \sigma_{t-1}^{V_2 \otimes V_3}$ and let $p_k =\deg(B_{k+1}^1(T))$ for $1\leq k\leq t-1$. In particular $\scrC^1_k(T)$ is well-defined for $k \leq t-1$. First, observe that if $p_k = k+1$, then $\scrC^1_k(T)$ is a single point: indeed, in this case, the ideal of $(k+1) \times (k+1)$ minors is generated by a single element. Moreover, in this case it is easy to show that $\bbP (T(V_1^*)) \subseteq \sigma_{k+1}^{V_2 \otimes V_3}$.

Therefore, in the following, assume $p_k \leq k$ for every $k \leq t-1$; we are going to show $p_k = k$. We proceed by induction on $k$. The base of the induction follows from the hypothesis because $\deg(B^1_2(T)) \leq 1$ and $B^1_2(T) \neq \emptyset$. The induction hypothesis guarantees $k-1\leq p_k\leq k$ because $p_{k-1} \leq p_k$. Since the linear space of partials of the minors of size $k$ is contained in the space of minors of size $k-1$, the intersection multiplicity of a point $q\in B_k^1(T)$, regarded as a point on $B^1_{k+1}(T)$ increases at least by one. This proves $p_k=k$. Geometrically, the initial single intersection point in $B_2^1(T)$ gets fatter and fatter along the intersection with higher degree minors. 
\end{proof}

\section{Nets}\label{sec: nets}
If $\dim V_1 = 3$, then the image of the first flattening of a concise tensor $T \in V_1 \otimes V_2 \otimes V_3$ is regarded, projectively, as a $2$-dimensional space of matrices. This is the case of nets of matrices. In this section, we characterize the collineation varieties of nets of matrices when $\dim V_2, \dim V_3$ are small. In particular, we will consider the cases $n_2 = 2$ or $(n_2,n_3) = (3,3)$.

In these cases, the only collineation varieties of interest are $\scrC^i_2(T)$. In particular, all collineation varieties arise as images of $\bbP^2$ via a system of quadrics. Following the construction of \autoref{subsec: linear series and projections}, we record the following result. 
\begin{lemma}\label{lemma: classification blow up P2}
Let $L \subseteq S^2 \bbC^3$ be a linear subspace of quadrics of dimension $m+1 \geq 2$ with $0$-dimensional base locus $B \subseteq \bbP^2$. Let
    \[
    \phi_L : \bbP^2 \dashto \bbP L^*
    \]
    be the induced map and let $X_L$ be the closure of its image. If $m = 5-\deg(B)$, then $X_L$ coincides with one of the following:
\begin{itemize}
    \item if $B = \emptyset$, then $X_L = \nu_2(\bbP^2)$;
    \item if $B $ is a single point, then $X_L = S_{(1,2)}$ is the cubic scroll in $\bbP^4$;
    \item if $B $ is a set of two points, then $X_L = \bbP^1 \times \bbP^1$ is the quadric surface in $\bbP^3$;
    \item if $B $ is a local scheme of length $2$, then $X_L = S_{(0,2)}$ is the cone over a conic;
    \item if $B$ is a scheme of length $3$ different from a fat point of multiplicity $2$, then $X_L = \bbP^2$ is a linear space;
    \item if $B$ is a fat point of multiplicity $2$, then $X_L$ is a plane conic curve;
    \item if $B$ is a scheme of length $4$, then $X_L$ is the projective line $\bbP^1$.
\end{itemize}
\end{lemma}
\begin{proof}
  Since $B$ is $0$-dimensional, by Bertini's Theorem two generic elements of $L$ cut out a $0$-dimensional scheme $\bbX$, with $\deg \bbX = 4$ and $\dim I(\bbX)_2 = 2$. If $\deg(B) = 4$, then $B = \bbX$. In this case $\dim L = 2$, and $\phi_L$ is surjective onto $\bbP L^* = \bbP^1$.

 In all other cases, $B$ is a scheme of degree at most $3$. We first observe $\dim I(B)_2 = 6-\deg(B)$. This is a consequence of basic properties of the Hilbert function of $0$-dimensional schemes: since $I(B)$ is generated by forms of degree $2$, the Hilbert function of $B$ satisfies $h_B(2) = \deg(B)$, or equivalently $\dim I(B)_2 = 6-\deg(B)$. The condition $m = 5-\deg(B)$ is equivalent to $L = I(B)_2$.
 
Since $B$ is a scheme of degree at most $3$, it can be normalized using the action of $\SL_3$ on $\bbP^2$. Implicitly, this classification is recorded in \cite{BucLan:ThirdSecantVar}. The proof is obtained by computing $X_L$ for every possible normalized $B$. We use coordinates $z_0 \vvirg z_m$ on $\bbP L^*$.
\begin{itemize}
    \item If $B = \emptyset$, the statement is clear.
    \item If $B = \{ p_0\}$ is a single point, then we may assume $p_0 = (1,0,0)$. In this case, we have $\phi_L (x_0,x_1,x_2) = ( x_0x_1,x_0x_2 , x_1^2 ,x_1x_2,x_2^2)$. This gives the standard parameterization of the cubic scroll $S_{(1,2)}$ in $\bbP^4$.
    \item If $B = \{ p_1,p_2\}$ is a set of two points, assume $p_1 = (0,1,0)$ and $p_2 = (0,0,1)$. Then $\phi_L (x_0,x_1,x_2) = (x_0^2,x_0x_1,x_0x_2,x_1x_2)$ satisfies the single quadratic equations $z_0z_3-z_1z_2=0$. This shows that $X_L$ is the quadric surface in $\bbP^3$. 
    \item If $B$ is a scheme of length $2$, we may assume $I(B) = (x_1^2,x_2)$. Considering the degree $2$ component of $I(B)$, we have $\phi_L (x_0,x_1,x_2) = (x_1^2,x_0x_2,x_1x_2,x_2^2)$. This satisfies the quadratic equation $z_0z_2-z_1^2=0$, therefore $X_L$ is a quadric cone.
    \item If $B$ is a scheme of length $3$ different from a fat point of multiplicity $2$, we have three possible cases, depending on whether $B$ is supported at one, two, or three points. In these three cases $\dim I(B)_2 = 3$; to show that $X_L = \bbP^2$, we observe that the three quadric generators have no polynomial dependencies. By semicontinuity, it is enough to do this in the case where $B$ is supported at a single point: in this case $I(B)$ is a curvilinear scheme, which may be normalized so that $\phi_L (x_0,x_1,x_2) = (x_2^2,x_1x_2,x_1^2-x_0x_2)$. The differential of $\phi_L$ at $(0,0,1)$ is full rank, showing that $\dim X_L = 2$, hence $X_L = \bbP^2$.
    \item If $B$ is a fat point of multiplicity $2$, then we may assume its support is $(1,0,0)$, so that $\phi_L (x_0,x_1,x_2) = (x_1^2,x_1x_2,x_2^2)$. This parameterizes the plane conic having equation $z_0z_2-z_1^2$.
\end{itemize}
\end{proof}

In order to study collineation varieties of nets, we rely on classification results for tensors in $\bbC^3 \otimes \bbC^2 \otimes \bbC^{n_3}$ and $\bbC^3 \otimes \bbC^3 \otimes \bbC^3$. For the first one, we resort once again to Kronecker's classification of pencils of matrices. The second one has been extensively studied, in the context of the geometry of elliptic curves \cite{Ng:Classification333trilinearForms} and in geometric invariant theory \cite{Nur:OrbitsInvariantsCubicMatrices,Nur:ClosuresNilpotent}. We use the classification recorded recently in \cite{DitDeGrMar:ClassificationThreeQutrit}, which is built on \cite{Nur:OrbitsInvariantsCubicMatrices,Nur:ClosuresNilpotent}.

\subsection{Nets in $\bbC^2 \otimes \bbC^{n_3}$}\label{C2C3}

In this section, we regard tensors in $V_1 \otimes V_2 \otimes V_3$ with $(n_1,n_2,n_3) = (3,2,n_3)$ as linear spaces of matrices of size $2 \times n_3$. If  $T$ is concise in the $\bbC^3$ factor, then this is a net of $2 \times n_3$ matrices, and one can consider the map defined by its $2 \times 2$ minors:
\[
\mu^{1}_{k,T} : \bbP {\bbC^3}^* \dashto \bbP( \Lambda^2 \bbC^2 \otimes \Lambda^2 \bbC^{n_3}) \simeq \bbP (\Lambda^2 \bbC^{n_3}).
\]
After possibly restricting the $\bbC^{n_3}$ factors, we may assume $n_3\leq 6$. The space $\bbC^3 \otimes \bbC^2 \otimes \bbC^6$ has finite representation type in the sense of Gabriel's Theorem \cite{Gabriel-quivers}. In particular, it is union of $26$ orbits for the action of $\GL_3 \times \GL_2 \times \GL_6$. These orbits are recorded in \cite[Sec. 10.3]{Lan:TensorBook} as pencils of $3 \times 6$ matrices. 

Orbits 1--9, 11, 12 in the classification of \cite[Sec. 10.3]{Lan:TensorBook} are not concise in the $\bbC^3$ factor; hence they do not give rise to a net and in this case the image of the map $\mu^{1}_{k,T}$ has dimension at most $1$. The result in these cases can be deduced using \autoref{thm: pencils general}. 

Orbit 10 is not concise on the $\bbC^2$ factor, hence $\mu^{1}_{k,T}$ is not defined. 

The results for Orbits 13--26 are recorded in the following: 
\begin{theorem}\label{thm: small nets}
    Let $T \in  V_1 \otimes V_2 \otimes V_3$ be a concise tensor with $\dim V_1 = 3$, $\dim V_2 = 2$. Then, $\scrC^1_{2}(T)$ is one of the following:
    \begin{itemize}
        \item The projective line $\bbP^1$: this occurs for orbits 14, 15, 18;
        \item The linear space $\bbP^2$: this occurs for orbits 13, 16, 17, 20;
        \item A quadric cone in $\bbP^3$: this occurs for orbit 21;
        \item A quadric surface in $\bbP^3$: this occurs for orbit 22;
        \item A rational cubic scroll $S_{(1,2)}$ in $\bbP^4$; this occurs for orbits 19, 24;
        \item A Veronese surface $\nu_2(\bbP^2)$ in $\bbP^5$; this occurs for orbits 23, 25, 26.
    \end{itemize}
\end{theorem}
\begin{proof}
Let $L \subseteq S^2 V_1$ be the linear space of quadrics defined by the $2 \times 2$ minors of the net of quadrics corresponding to each orbit. In the case of orbits 14, 15, and 18, $\dim L = 2$, hence the image is $\bbP^1$. For orbits 13, 20, $\dim L = 3$ and its base locus $B_L$ is $1$-dimensional; in this case, the collineation variety is $\bbP^2$.

In all other cases, a direct \texttt{Macaulay2} computation shows that the base locus $B_L$ is $0$-dimensional (or empty) and verifies the equality $m=5-\deg(B_L)$. Therefore the assumption of \autoref{lemma: classification blow up P2} is satisfied and implies the result. The code to perform the computation is available in the supplementary file \texttt{CollineationVariety.m2}. For each orbit, the script computes the dimension and degree of the base locus, and of its radical as well as the value of $\dim L$. 
\end{proof}

\subsection{Nets in $\bbC^3 \otimes \bbC^3$}

Let $V_1,V_2,V_3 \simeq \bbC^3$ and let $T \in  V_1 \otimes V_2 \otimes V_3$ be a concise tensor. The tensor $T$ naturally defines three nets of $3 \times 3$ matrices in three different ways 
\begin{align*}
T(V_1^*) \subseteq V_2 \otimes V_3, \qquad T(V_2^*) \subseteq V_3 \otimes V_1, \qquad T(V_3^*) \subseteq V_1 \otimes V_2.
\end{align*}
We consider the first linear space $T(V_1^*)$, and study the corresponding collineation varieties $\scrC^1_{2}(T) \subseteq \bbP (\Lambda^2 V_2 \otimes \Lambda^2 V_3) \simeq \bbP^8$. Since the parameterization is given by quadrics, $\scrC^1_2(T)$ is a projection of the Veronese variety $\nu_2(\bbP V_1^*)$ as discussed in \autoref{subsec: linear series and projections}.

The orbit structure of $\bbP( V_1  \otimes V_2 \otimes V_3)$ under the action of $\SL(V_1) \times \SL(V_2) \times \SL(V_3)$ is \emph{tame}, in the sense of Gabriel's Theorem \cite{Gabriel-quivers}. There is a two parameter family of generic tensors, some one parameter families and a large number of single orbits. The explicit orbit classification is the object of a long series of works \cite{Ng:Classification333trilinearForms,Nur:OrbitsInvariantsCubicMatrices,Nur:ClosuresNilpotent}. We refer to \cite{DitDeGrMar:ClassificationThreeQutrit} for a recent, complete classification. We point out that the collineation variety is invariant under a swap of the second and third tensor factors. Therefore, we use the classification of \cite{DitDeGrMar:ClassificationThreeQutrit} to deduce normal forms under the action of $\SL(V_1) \times [(\SL(V_2) \times \SL(V_3)) \rtimes \bbZ_2]$ where $\bbZ_2$ acts by swapping the factors $V_2$ and $V_3$ via a fixed isomorphism between the two spaces.

We follow  \cite{Ng:Classification333trilinearForms}  for the description of the $2$-parameter family of generic tensors. The three nets $T(V_i^*)$ naturally define three ternary cubics $\det(T(V_i^*)) \in S^3 V_i^*$ where $T(V_i^*)$ is regarded as a matrix of linear forms on $V_i^*$. Accordingly, these cubics define three plane curves $E_T^i \subseteq \bbP V_i^*$, whose geometry plays an important role in the classification. It turns out that if one of these curves is smooth, then the other two are smooth as well, and they are isomorphic as elliptic curves. In this case \cite{Ng:Classification333trilinearForms} calls the corresponding tensor a \emph{smooth cuboid}; if the curves are singular, then the tensor is a \emph{singular cuboid}. 
\begin{proposition}[Classification of smooth cuboids {\cite{Ng:Classification333trilinearForms,Ish:PositivePropCubicCurves}}]\label{thm: Ng classification}
    Let $T \in V_1 \otimes V_2 \otimes V_3$ be a tensor such that $E_T^1$ is smooth. Then, after a change of coordinates, $T(V_1^*)$ has the form
    \[
    T(V_1^*) = \left(\begin{array}{ccc}
    -p_1 x_0+x_1 & -p_2 x_0+x_2& 0 \\ 
    p_2 x_0+x_2 &  p_1^2 x_0+a x_0+p_1 x_1 & x_1 \\
    0 & x_1 & -x_0 
    \end{array}\right)
    \]
    where $E_T^1 = \{ x_0 x_2^2 - x_1^3 - a x_0^2x_1 - b x_0^3= 0 \} $, $b = p_1^3-p_2^2+p_1 a$,  and $(1:p_1:p_2) \in E^1_T$.
\end{proposition}

Smooth cuboids form an open set in the space $\bbP (V_1 \otimes V_2 \otimes V_3)$, given by the complement the discriminant of any of the three curves $E_T^i$. We prove that if a tensor $T$ belongs to this open set, then its collineation varieties behave as generic as well, and it coincides with the Veronese surface $\nu_2(\bbP^2)$. 
\begin{theorem}\label{thm: nets smooth}
     Let $T \in V_1 \otimes V_2 \otimes V_3$ be a tensor such that $E_T^1$ is a smooth elliptic curve. Then, the collineation variety $\scrC^i_2(T)$ is the Veronese surface $\nu_2(\bbP^2)$.
\end{theorem}
\begin{proof}
    The proof is purely computational. Without loss of generality, we may assume $i = 1$, and the tensor $T$ can be normalized so that $T(V_1^*)$ is the linear space described in \autoref{thm: Ng classification}, where $a,p_1,p_2$ are treated as parameters. The linear space generated by $2 \times 2$ minors of $T(V_1^*)$ has dimension $6$ for every choice of $a,p_1,p_2$, unless 
    \[
    p_2 = 0 \qquad \text{and} \qquad a + 3 p_1^2 = 0.
    \]
    In this case, the curve $E_T^1$ has equation $x_2^2x_0 - (p_1x_0-x_1)^2(2p_1x_0+x_1)$ which is singular. This computation is performed in \Macaulay, and the code is available in the supplementary file \texttt{CollineationVariety.m2}. 
    \end{proof}

The classification of singular cuboids is more delicate. As mentioned above, we follow \cite{DitDeGrMar:ClassificationThreeQutrit}: however only a subset of the normal forms are of interest for us, because we work projectively and the collineation variety is invariant under the action of the group $\bbZ_2$ swapping the second and third factor. The resulting classification has three $1$-parameters families of orbits and 39 additional orbits. \autoref{table: semistable orbits}  and \autoref{table:unstable} record, respectively, representatives of the semistable and the unstable orbits, in the sense of geometric invariant theory. This is a subset of the normal forms recorded in \cite[Table I,III,IV,V]{DitDeGrMar:ClassificationThreeQutrit}.
\begin{longtable}{c|c|c} 
   \cite{DitDeGrMar:ClassificationThreeQutrit} & tensor $T$ & $\scrC^1_2(T)$ \\  \midrule
III.1 & $\begin{smallmatrix}
    (a_0b_0c_0 + a_1b_1c_1 + a_2b_2c_2) + \\ 
    \lambda (a_0b_1c_2 + a_2b_0c_1 + a_1b_2c_0) + \\
     a_0b_2c_1 + a_1b_0c_2
\end{smallmatrix}$ & $\nu_2(\bbP^2)$ \\[10pt]
III.2 & $\begin{smallmatrix}
    (a_0b_0c_0 + a_1b_1c_1 + a_2b_2c_2) + \\ 
    \lambda (a_0b_1c_2 + a_2b_0c_1 + a_1b_2c_0) + \\
     a_0b_2c_1 
\end{smallmatrix}$  & $\nu_2(\bbP^2)$  \\[10pt]
III.ps & $\begin{smallmatrix}
    (a_0b_0c_0 + a_1b_1c_1 + a_2b_2c_2) + \\ 
    \lambda (a_0b_1c_2 + a_2b_0c_1 + a_1b_2c_0) 
\end{smallmatrix}$ & $\nu_2(\bbP^2)$  \\ [10pt]
IV.1(i) & $\begin{smallmatrix}
    a_0b_0c_0 + a_1b_1c_1 + a_2b_2c_2 + \\ 
    a_0b_1c_2 + a_0b_2c_1 + a_1b_0c_2 + a_1b_2c_0
\end{smallmatrix}$  & $S_{(1,2)}$ \\[10pt]
IV.1(ii) & $\begin{smallmatrix}
    a_0b_0c_0 + a_1b_1c_1 + a_2b_2c_2 + \\ 
    a_1b_0c_2 + a_2b_0c_1 + a_0b_1c_2 + a_2b_1c_0
\end{smallmatrix}$  & $\nu_2(\bbP^2)$ \\[10pt]
IV.1(iii) & $\begin{smallmatrix}
    a_0b_0c_0 + a_1b_1c_1 + a_2b_2c_2 + \\ 
    a_2b_1c_0 + a_1b_2c_0 + a_2b_0c_1 + a_0b_2c_1
\end{smallmatrix}$  & $\nu_2(\bbP^2)$ \\[10pt]
IV.2(i) & $\begin{smallmatrix}
    a_0b_0c_0 + a_1b_1c_1 + a_2b_2c_2 + \\ 
    a_0b_1c_2 + a_0b_2c_1 + a_1b_0c_2
\end{smallmatrix}$  & $S_{(1,2)}$ \\[10pt]
IV.2(iii) & $\begin{smallmatrix}
    a_0b_0c_0 + a_1b_1c_1 + a_2b_2c_2 + \\ 
    a_1b_0c_2 + a_2b_0c_1 + a_0b_1c_2
\end{smallmatrix}$  & $\nu_2(\bbP^2)$ \\[10pt]
IV.2(vi) & $\begin{smallmatrix}
    a_0b_0c_0 + a_1b_1c_1 + a_2b_2c_2 + \\ 
    a_2b_1c_0 + a_1b_2c_0 + a_2b_0c_1
\end{smallmatrix}$  & $S_{(1,2)}$\\[10pt]
IV.4(i) & $\begin{smallmatrix}
    a_0b_0c_0 + a_1b_1c_1 + a_2b_2c_2 + \\ 
    a_0b_1c_2 + a_0b_2c_1
\end{smallmatrix}$  & $\bbP^1 \times \bbP^1$\\[10pt]
IV.4(ii) & $\begin{smallmatrix}
    a_0b_0c_0 + a_1b_1c_1 + a_2b_2c_2 + \\ 
    a_1b_0c_2 + a_2b_0c_1
\end{smallmatrix}$  & $S_{(1,2)}$\\[10pt]
IV.4(iii) & $\begin{smallmatrix}
    a_0b_0c_0 + a_1b_1c_1 + a_2b_2c_2 + \\ 
    a_2b_1c_0 + a_1b_2c_0
\end{smallmatrix}$  & $S_{(1,2)}$\\[10pt]
IV.5(i) & $\begin{smallmatrix}
    a_0b_0c_0 + a_1b_1c_1 + a_2b_2c_2 + \\ 
    a_0b_1c_2 + a_1b_2c_0
\end{smallmatrix}$ & $S_{(1,2)}$\\[10pt]
IV.5(ii) & $\begin{smallmatrix}
    a_0b_0c_0 + a_1b_1c_1 + a_2b_2c_2 + \\ 
    a_1b_0c_2 + a_2b_1c_0 
\end{smallmatrix}$ & $S_{(1,2)}$\\[10pt]
IV.7(i) & $\begin{smallmatrix}
    a_0b_0c_0 + a_1b_1c_1 + a_2b_2c_2 + \\ 
    a_0b_1c_2 
\end{smallmatrix}$ & $\bbP^1 \times \bbP^1$\\[10pt]
IV.ps & $\begin{smallmatrix}
    a_0b_0c_0 + a_1b_1c_1 + a_2b_2c_2 
\end{smallmatrix}$ & $\bbP^2$\\[10pt]
V.1 & $\begin{smallmatrix}
  a_0b_1c_2 + a_1b_2c_0 + a_2b_0c_1 + \\
  -(a_0b_2c_1 + a_1b_0c_2 + a_2b_1c_0) + \\
  a_0b_0c_0 + a_1b_1c_2 + a_1b_2c_1 + a_2b_1c_1
\end{smallmatrix}$ & $\nu_2(\bbP^2)$ \\[20pt]
V.2 & $\begin{smallmatrix}
  a_0b_1c_2 + a_1b_2c_0 + a_2b_0c_1 + \\
  -(a_0b_2c_1 + a_1b_0c_2 + a_2b_1c_0) + \\
a_0b_0c_1 + a_0b_1c_0 + a_1b_0c_0 + \\
a_1b_1c_2 + a_1b_2c_1 + a_2b_1c_1
\end{smallmatrix}$  & $\nu_2(\bbP^2)$ \\[20pt]
V.3 & $\begin{smallmatrix}
  a_0b_1c_2 + a_1b_2c_0 + a_2b_0c_1 + \\
  -(a_0b_2c_1 + a_1b_0c_2 + a_2b_1c_0) + \\
  a_0b_0c_0 + a_1b_1c_1
\end{smallmatrix}$ & $\nu_2(\bbP^2)$ \\[20pt]
V.4 & $\begin{smallmatrix}
  a_0b_1c_2 + a_1b_2c_0 + a_2b_0c_1 + \\
  -(a_0b_2c_1 + a_1b_0c_2 + a_2b_1c_0) + \\
  a_0b_0c_1 + a_0b_1c_0 + a_1b_0c_0
\end{smallmatrix}$& $\nu_2(\bbP^2)$ \\[20pt]
V.5 & $\begin{smallmatrix}
  a_0b_1c_2 + a_1b_2c_0 + a_2b_0c_1 + \\
  -(a_0b_2c_1 + a_1b_0c_2 + a_2b_1c_0) + \\
  a_0b_0c_0
\end{smallmatrix}$&  $\nu_2(\bbP^2)$ \\[20pt]
V.ps & $\begin{smallmatrix}
  a_0b_1c_2 + a_1b_2c_0 + a_2b_0c_1 + \\
  -(a_0b_2c_1 + a_1b_0c_2 + a_2b_1c_0)
  \end{smallmatrix}$& $\nu_2(\bbP^2)$ \\
    \caption{The \cite{DitDeGrMar:ClassificationThreeQutrit} classification of semistable orbits, with the corresponding second collineation variety. The classification is reduced to projective space and modulo the swap of the second and third factor. The terms marked with `ps' are the polystable elements of each family, which do not appear in the corresponding table of \cite{DitDeGrMar:ClassificationThreeQutrit}. The parameter $\lambda$ in the first three orbits is assumed to be non-zero.}
    \label{table: semistable orbits}
\end{longtable}

\begin{longtable}{c|c|c} 
   \cite{DitDeGrMar:ClassificationThreeQutrit} & tensor $T$ & $\scrC^1_2(T)$ \\  \midrule
I.1(i) &  $\begin{smallmatrix} 
 a_0b_1c_2 + a_0b_2c_1 + a_1b_0c_2 + \\
     a_1b_1c_1 + a_1b_2c_0 + a_2b_0c_0 \end{smallmatrix}  $ & $S_{(1,2)}$ \\[10pt] 
I.1(ii) &  $\begin{smallmatrix} 
 a_1b_0c_2 + a_2b_0c_1 + a_0b_1c_2 + \\
     a_1b_1c_1 + a_2b_1c_0 + a_0b_2c_0 \end{smallmatrix}  $ & $\nu_2(\bbP^2)$ \\[10pt] 
I.2(i) &  $\begin{smallmatrix} 
 a_0b_1c_2 + a_0b_2c_1 + a_1b_0c_2 + \\
     a_1b_1c_0 + a_1b_1c_1 + a_2b_0c_0 \end{smallmatrix}  $ & $S_{(1,2)}$ \\[10pt] 
I.2(iii) &  $\begin{smallmatrix} 
 a_2b_0c_1 + a_1b_0c_2 + a_2b_1c_0 + \\
     a_0b_1c_1 + a_1b_1c_1 + a_0b_2c_0  \end{smallmatrix}  $ & $\nu_2(\bbP^2)$ \\[10pt] 
I.3(i) &  $\begin{smallmatrix} 
 a_0b_0c_2+a_0b_1c_1+a_0b_2c_0 + \\
     a_1b_0c_1+a_1b_1c_2+a_2b_0c_0  \end{smallmatrix}  $ & $S_{(1,2)}$ \\[10pt] 
I.3(iii) &  $\begin{smallmatrix} 
 a_2b_0c_0+a_1b_0c_1+a_0b_0c_2 + \\
     a_1b_1c_0+a_2b_1c_1+a_0b_2c_0 \end{smallmatrix}  $ & $\bbP^1 \times \bbP^1$ \\[10pt] 
I.4(i) &  $\begin{smallmatrix} 
 a_0b_0c_2 + a_0b_1c_1 + a_1b_0c_1 + \\
     a_1b_1c_0 + a_2b_2c_0  \end{smallmatrix}  $ & $S_{(1,2)}$ \\[10pt] 
I.4(iii) &  $\begin{smallmatrix} 
 a_2b_0c_0 + a_1b_0c_1 + a_1b_1c_0 + \\
     a_0b_1c_1 + a_0b_2c_2 \end{smallmatrix}  $ & $S_{(0,2)}$ \\[10pt] 
I.5(i) &  $\begin{smallmatrix} 
 a_0b_0c_2 + a_0b_2c_0 + a_0b_2c_1 + \\
     a_1b_1c_0 + a_2b_0c_1 \end{smallmatrix}  $ & $\bbP^1 \times \bbP^1$ \\[10pt] 
I.5(iii) &  $\begin{smallmatrix} 
 a_2b_0c_0 + a_0b_0c_2 + a_1b_0c_2 + \\
     a_0b_1c_1 + a_1b_2c_0 \end{smallmatrix}  $ & $S_{(0,2)}$ \\[10pt] 
I.6(i) &  $\begin{smallmatrix} 
 a_0b_0c_2 + a_0b_1c_1 + a_1b_0c_1 + \\
     a_1b_2c_0 + a_2b_1c_0  \end{smallmatrix}  $ & $S_{(1,2)}$ \\[10pt] 
I.6(iii) &  $\begin{smallmatrix} 
 a_2b_0c_0 + a_1b_0c_1 + a_1b_1c_0 + \\
     a_0b_1c_2 + a_0b_2c_1  \end{smallmatrix}  $ & $S_{(0,2)}$ \\[10pt] 
I.7(i) &  $\begin{smallmatrix} 
 a_0b_0c_2 + a_0b_1c_1 + a_0b_2c_0 + \\
     a_1b_0c_1 +  a_2b_1c_0  \end{smallmatrix}  $ & $\bbP^1 \times \bbP^1$ \\[10pt] 
I.7(ii) &  $\begin{smallmatrix} 
 a_0b_0c_2 + a_1b_0c_1 + a_2b_0c_0 + \\
     a_0b_1c_1 +  a_1b_2c_0  \end{smallmatrix}  $ & $S_{(0,2)}$ \\[10pt] 
I.8 &  $\begin{smallmatrix} 
 a_0b_0c_2 + a_0b_2c_0 + a_1b_1c_1 + \\
     a_2b_0c_0  \end{smallmatrix}  $ & $\bbP^2$ \\[10pt] 
I.10 &  $\begin{smallmatrix} 
 a_0b_0c_2 + a_0b_1c_1 + a_0b_2c_0 + \\
     a_1b_0c_1 + a_1b_1c_0 + a_2b_0c_0 \end{smallmatrix}  $ & $\bbP^2$ \\[10pt] 
I.11(i) &  $\begin{smallmatrix} 
 a_0b_0c_2 + a_0b_2c_0 + a_1b_0c_1 + \\
    a_2b_1c_0 \end{smallmatrix}  $ & $\bbP^1 \times \bbP^1$ \\[10pt] 
I.11(ii) &  $\begin{smallmatrix} 
 a_0b_0c_2 + a_2b_0c_0 + a_0b_1c_1 + \\
    a_1b_2c_0 \end{smallmatrix}  $ & $\bbP^2$ \\[10pt] 
I.14(i) &  $\begin{smallmatrix} 
 a_0b_0c_2 + a_0b_1c_0 + a_0b_2c_1 + \\
    a_1b_0c_0 + a_2b_0c_1 \end{smallmatrix}  $ & $\bbP^2$ \\[10pt] 
I.14(iii) &  $\begin{smallmatrix} 
 a_2b_0c_0 + a_0b_0c_1 + a_1b_0c_2 + \\
    a_0b_1c_0 + a_1b_2c_0 \end{smallmatrix}  $ & plane conic \\[10pt] 
\caption{The \cite{DitDeGrMar:ClassificationThreeQutrit} classification of unstable orbits, with the corresponding second collineation variety. The classification is reduced to projective space and modulo swapping the second and third factor} \label{table:unstable}
\end{longtable}

\begin{theorem}\label{thm: netC3}
Let $T \in V_1 \otimes V_2 \otimes V_3$, with $\dim V_i = 3$ for $i = 1,2,3$. If the cubic curve $E^1_T$ is singular, then $T$ is, up to changing coordinates and swapping the second and third factor, one of the tensors in \autoref{table: semistable orbits} or \autoref{table:unstable}. The corresponding collineation varieties on the first factor are recorded in the table.
\end{theorem}
\begin{proof}
The result of \cite{DitDeGrMar:ClassificationThreeQutrit} guarantees that \autoref{table: semistable orbits} and \autoref{table:unstable} cover all the orbits for the action of $\SL(V_1) \times [\SL(V_2) \times \SL(V_3)] \rtimes \bbZ_2$. 

A direct \texttt{Macaulay2} computation shows that in cases I.11.(ii) and I.14(i), the base locus $B_2^1(T)$ is $1$-dimensional: in these cases an explicit calculation shows $\scrC^1_2(T) = \bbP^2$. In all other cases, $B_2^1(T)$ is $0$-dimensional for which the condition $m=5-\deg(B_2^1(T))$ in \autoref{lemma: classification blow up P2} holds and the collineation varieties are obtained by applying this Lemma.  The code to perform the computation is available in the supplementary file \texttt{CollineationVariety.m2}. For each orbit, the script computes the dimension and degree of the base locus and of its radical as well as the value of $\dim L$. 
\end{proof}

\section{Stratification of tensor spaces via collineation varieties}\label{sec: stratification}

In this section, we introduce a weak version of a stratification of tensor spaces in terms of the collineation varieties. Let $V_1, V_2,V_3$ be vector spaces, $k$ be an integer, $i =1,2,3$ be an integer, and let $X \subseteq \bbP^M$ be an embedded algebraic variety. Define
\[
{\scrF^{i}_{k,X}}^\circ = \{ T \in V_1 \otimes V_2 \otimes V_3 : T \text{ is concise and } \scrC^i_{k}(T) \simeq X\}.
\]
The relation $\scrC^i_{k}(T) \simeq X$ is a linear isomorphism as embedded varieties. From the discussion in \autoref{subsec: linear series and projections}, all varieties $X$ for which ${\scrF^{i}_{k,X}}^\circ$ is nonempty arise as projections of the Veronese variety $\nu_k ( \bbP V_i^*)$. In particular, one can prove that ${\scrF^{i}_{k,X}}^\circ$ is open in its closure, hence it is quasi-projective. The closures $\scrF^{i}_{k,X} = \bar{{\scrF^{i}_{k,X}}^\circ}$ are, often reducible, subvarieties of $\bbP (V_1 \otimes V_2 \otimes V_3)$. For every $k$ and $i$, they define a weak version of a stratification of $\bbP (V_1 \otimes V_2 \otimes V_3)$ in the sense that for any two non-isomorphic varieties $X_1,X_2$ we have that
\begin{itemize}
    \item the intersection of the open strata is empty: ${\scrF^i_{k,X_1}}^\circ \cap {\scrF^i_{k,X_2}}^\circ$;
    \item the intersection of the closed strata is covered by other strata: $\scrF^i_{k,X_1} \cap \scrF^i_{k,X_2} \subseteq \bigcup_{Y \in \calH} \scrF^i_{k,Y}$, where $\calH$ is a suitable (possibly infinite) index set with the property that $\scrF^i_{k,Y} \not \supseteq \scrF^i_{k,X_j}$ for $j=1,2$ and for every $Y \in \calH$.
\end{itemize}
In a stratification, one would usually require a stronger second property, namely that the intersection of closed strata is exactly the union of other strata, possibly infinitely many. This is indeed what happens in the case of $\bbP (\bbC^2 \otimes \bbC^{n_2} \otimes \bbC^{n_3})$ and $\bbP(\bbC^3 \otimes \bbC^3 \otimes \bbC^3)$. In general, we do not expect it to be the case.

In the case of pencils, the strata are uniquely determined by the degree of the corresponding base locus and the Kronecker classification guarantees that they are totally ordered by inclusion. Similarly, following the classification result of \autoref{sec: nets}, one can verify that the strata in $\bbP (\bbC^3 \otimes \bbC^3 \otimes \bbC^3)$ are totally ordered.

However, in general, we expect the stratification to be infinite and difficult to characterize. One can coarsen it, considering certain algebraic or geometric invariants of $X$, such as its dimension, its degree, or the dimension of its linear span. For instance, only considering the degree of $X$ retrieves the stratification given by the characteristic numbers studied in \cite{ConMic:CharNumChromNumTensors}.

\subsection{Stratification of spaces of pencils}\label{subsec: stratification pencils}

\autoref{thm: pencils general} shows that the collineation varieties of a matrix pencil are always rational normal curves, whose degree depends on the intersection multiplicity of the pencil with the varieties of low-rank matrices.

Therefore, in this case, we have the following immediate consequence of \autoref{thm: pencils general}:
\begin{corollary}\label{corol: pencil strata}
  Let $s \leq k \leq n_2 \leq n_3$ be non-negative integers. Let $X_s \subseteq \bbP^s$ be the rational normal curve of degree $s$. Then, 
  \[
 \scrF^1_{k,X_{s}} = \bar{ \{ T \in \bbC^2 \otimes \bbC^{n_2} \otimes \bbC^{n_3} : T \text{ is concise and } \deg B_k^1(T) = k-s\} }.
  \]
\end{corollary}
In general even understanding the irreducible decomposition of $\scrF^1_{k,X_s}$ seems non-trivial. The next result provides some information about this, generalizing \cite[Lemma 5.1]{BerDLaGes:DimensionTNS}.
\begin{proposition}\label{prop: small Fs for pencils}
    Let $s \leq k \leq n_2 \leq n_3$ be non-negative integers with $n_2 < n_3$ or $k < n_2 = n_3$. If $k-s = 0,1$ then $\scrF^1_{k,X_s}$ is irreducible. If $k-s = 2$, then $\scrF^1_{k,X_s}$ has exactly two irreducible components if $k\geq 3$, and is irreducible when $k = 2$. The dimensions are recorded in the following table:
    \[
    \begin{array}{c|c} 
    k-s & \dim \scrF^1_{k,X_s} \\ \midrule
    0 & 2n_2n_3-1\\ ~ \\
    1 & n_2n_3 + (k-1)( n_2 + n_3 - (k-1)) \\ ~\\ 
    \multirow{ 2}{*}{$2$} & 2 (k-1) (n_2 + n_3 - (k-1)), \\ & n_2n_3 + (k-2)( n_2 + n_3 - (k-2)) \\
    \end{array}
    \]
\end{proposition}
\begin{proof}
If $k-s = 0$, the statement is clear and indeed, in this case $\scrF^1_{k,X_s} = \bbP( \bbC^2 \otimes \bbC^{n_2} \otimes \bbC^{n_3})$. In the case $k-s = 1$, observe that 
\[
\scrF^1_{k,X_{s}} = \bar{ \{ T \in \bbC^2 \otimes \bbC^{n_2}\otimes \bbC^{n_3} : T \text{ is concise and } B_k^1(T) \neq \emptyset \} }.
\]
This is a consequence of \autoref{corol: pencil strata}. Indeed, under the assumption $k < n_2 = n_3$ or $n_2 < n_3$, a generic line $L \subseteq \bbP (\bbC^{n_2}\otimes \bbC^{n_3})$ does not intersect $\sigma_{k-1}$. The lines that do intersect $\sigma_{k-1}$ form an irreducible subvariety of the Grassmannian $\calS \subseteq \Gr(2, \bbC^{n_2}\otimes \bbC^{n_3})$, whose generic element is a pencil intersecting $\sigma_{k-1}$ in a single point; this is a higher codimension analog of \cite[Prop. 2.2.2]{GKZ}. The same proof shows that
\begin{align*}
\dim \calS =& \dim \sigma_{k-1} + \dim  ( \bbP ( \bbC^{n_2}\otimes \bbC^{n_3} / \langle A \rangle)) = \\ 
&[(k-1)( n_2 + n_3 - (k-1)) - 1] + [n_2n_3 -2];
\end{align*}
here $ \bbP ( \bbC^{n_2}\otimes \bbC^{n_3} / \langle A \rangle)$ is identified with the fiber of the projection from $\calS$ to $\sigma_{k-1}$ over a generic element $A \in \sigma_{k-1}$. Similarly to \cite[Lemma 5.1]{BerDLaGes:DimensionTNS}, $\scrF^1_{k,X_{s}}$ defines, locally, a $\PGL_2$-bundle over the open subset of $\calS$ consisting of lines intersecting $\sigma_{k-1}$ at exactly one point. This shows that $\scrF^1_{k,X_{s}}$ is irreducible, of dimension $3 + \dim \calS = n_2n_3 + (k-1)( n_2 + n_3 - k +1)$. 

If $k-s = 2$, and $k\geq 3$ we show that $\scrF^1_{k,X_{s}}$ has exactly two components $\scrF^1_{k,X_{s}}=\bar{\Sigma}\cup \bar{Z}$ where
\begin{align*}
 \Sigma = & \{ T \in \bbC^2 \otimes \bbC^{n_2} \otimes \bbC^{n_3} : T \text{ is concise and }B_k^1(T) \text{ consists of two smooth points}\}   \\
Z=&  \{ T \in \bbC^2 \otimes \bbC^{n_2} \otimes \bbC^{n_3} : T \text{ is concise and } B_{k-1}^1(T) \neq \emptyset\} .
\end{align*}

Note that $Z$ is not defined if $k=2$. It is clear that $\Sigma \subseteq \scrF^1_{k,X_{s}}$. 

To show that the $Z \subseteq \scrF^1_{k,X_{s}}$, we need to show that a generic line $L \subseteq \bbP( V_2 \otimes V_3)$ passing through a generic element $A_0 \in \sigma_{k-2}$ intersects $\sigma_{k-1}$ with multiplicity exactly $2$. Consider first the following case: let $A'_0,A'_1$ be matrices of size $k$ with $A'_1$ generic and $\rank(A'_0) = k-2$. Then, $f(\lambda) = \det(A'_0 + \lambda A'_1)$ has a double root at $\lambda = 0$, showing that the line $\langle A'_0, A'_1\rangle$ intersects $\sigma_{k-1}$ with multiplicity two at $A'_0$. In the general case, consider $A_i$ to be $A_i'$ in a $k\times k$ upper left submatrix and $0$ elsewhere; then $f(\lambda)$ is one of the equations of $L \cap \sigma_{k-1}$, showing that the line $L$ spanned by $A_0$ and $A_1$ intersects $\sigma_{k-1}$ with multiplicity \emph{at most} $2$ at $A_0$, and the same holds for a generic line. Since the tangent cone to $\sigma_{k-1}$ at $A_0$ is linearly non-degenerate \cite[p.69, eqn.(2.2)]{ArCoGrHa:Vol1}, we conclude that such multiplicity is exactly $2$. We conclude that $\bar{\Sigma} \cup \bar{Z} \subseteq \scrF^1_{k,X_{s}}$. 

To show that equality holds, let $T \in \scrF^1_{k,X_{s}}$. By semicontinuity, $\deg B_k^1(T) \geq 2$. We have $B_{k-1}^1(T)\subseteq B_k^1(T)$. If $B_{k-1}^1(T)\neq \emptyset$, then $T\in \bar{Z}$, and otherwise $T\in \bar{\Sigma}$.

The dimension of $\Sigma$ was computed in \cite[Lemma 5.1]{BerDLaGes:DimensionTNS}: we have $\dim \Sigma = 2(k-1)(n_2+n_3-(k-1))$. On the other hand, since $k-1 - s = 1$, we have $Z = \scrF^1_{k-1,X_{s}}$ by \autoref{thm: pencils general} and the previous part of the proof. Therefore, $\dim Z = n_2n_3 + (k-2)( n_2 + n_3 - (k-2))$. This concludes the proof.
\end{proof}

If $(n_2,n_3) = (3,3),(3,4)$, then $\scrF^1_{2,\bbP^1}$ is the tensor network variety described in \cite[Theorem 5.2]{BerDLaGes:DimensionTNS}. This is the variety of tensors in $\bbC^2 \otimes \bbC^{n_2} \otimes \bbC^{n_3}$ that can be expressed as a degeneration of the $2 \times 2 \times 2$ matrix multiplication tensor; in particular, as a byproduct of \autoref{prop: small Fs for pencils}, we obtain some characteristic numbers of the $2 \times 2 \times 2$ matrix multiplication tensor.

\subsection{Stratification of spaces of small nets}\label{subsec: stratification nets} ~
In this section, we provide a characterization of the variety $\scrF_{2,X}^{1}$ in some cases of nets in $V_1 \otimes V_2 \otimes V_3$ with $\dim V_i = 3$, following the result of \autoref{thm: netC3}. 

Given a tensor $T \in V_1 \otimes V_2 \otimes V_3$, let $\Omega(T)$ be the orbit-closure of $T$ in $\bbP (V_1 \otimes V_2 \otimes V_3)$ under the action of $\GL(V_1) \times \GL(V_2) \times \GL(V_3)$. We denote tensors via the name of the corresponding family in \autoref{table: semistable orbits} and \autoref{table:unstable}. For instance, $T_{\mathrm{IV.ps}} = a_0 b_0c_0 + a_1b_1c_1 + a_2b_2c_2$ is the unit tensor of rank $3$ appearing \autoref{table: semistable orbits}.

The third secant variety of the Segre product is $\sigma_3(\bbP^2 \times \bbP^2 \times \bbP^2) = \Omega( T_{\mathrm{IV.ps}} )$. The set-theoretic equations for $\sigma_3(\bbP^2 \times \bbP^2 \times \bbP^2)$ are known: they are described in \cite{Lan:TensorBook} in the language of Young flattenings. More precisely, there is a $9 \times 9$ matrix $\Str(T)$, whose entries are linear forms of $T$, with the property that 
\[
\sigma_3(\bbP^2 \times \bbP^2 \times \bbP^2) = \{ T \in V_1 \otimes V_2 \otimes V_3 : \rank( \Str(T))< 7\}.
\]
The matrix $\Str(T)$ is called the \emph{Strassen flattening} of $T$. 

\begin{theorem}\label{thm: sigma3 as F1}
The variety $\scrF^1_{2,\bbP^2}$ has two components:
\[
\sigma_3(\bbP^2 \times \bbP^2 \times \bbP^2) \qquad \text{and} \qquad \Omega(T_{\mathrm{I.11.(ii)}}). 
\]
Moreover, $\sigma_3(\bbP^2 \times \bbP^2 \times \bbP^2) = \scrF^1_{2,\bbP^2} \cap \scrF^2_{2,\bbP^2}$.
\end{theorem}
\begin{proof}
From \autoref{table: semistable orbits} and \autoref{table:unstable}, we observe 
\begin{align*}
    \scrC^1_2(T_{\mathrm{IV.ps}}) =  \scrC^1_2(T_{\mathrm{I.11.(ii)}}) = \bbP^2,
\end{align*}
therefore $\sigma_3(\bbP^2 \times \bbP^2 \times \bbP^2) \cup \Omega(T_{\mathrm{I.11.(ii)}}) \subseteq \scrF^1_{2,\bbP^2}$. The semicontinuity of the degree of the collineation variety guarantees that the only elements in $\scrF^1_{2,\bbP^2}$ can be those with $\scrC^1_2(T) = \bbP^2$ or the tensor $T_{\mathrm{I.14.(iii)}}$ for which $\scrC^1_2(T)$ is a plane conic curve. One can verify directly using Strassen's flattening that all such elements are contained in $\sigma_3(\bbP^2 \times \bbP^2 \times \bbP^2)$. On the other hand,  $T_{\mathrm{I.11.(ii)}} \notin \sigma_3(\bbP^2 \times \bbP^2 \times \bbP^2)$, because its Strassen flattening has rank $8$. 

This shows that indeed $\scrF^1_{2,\bbP^2} = \sigma_3(\bbP^2 \times \bbP^2 \times \bbP^2) \cup \Omega(T_{\mathrm{I.11.(ii)}})$, and that these are two distinct irreducible components. Similarly, $\scrF^2_{2,\bbP^2} = \sigma_3(\bbP^2 \times \bbP^2 \times \bbP^2) \cup \Omega(T_{\mathrm{I.11.(i)}})$. 

We observe that $\Omega(T_{\mathrm{I.11.(ii)}}) \cap \Omega(T_{\mathrm{I.11.(i)}}) \subseteq \sigma_3(\bbP^2 \times \bbP^2 \times \bbP^2)$. To see this, consider the list of tensors in $\Omega(T_{\mathrm{I.11.(ii)}}) \cap \Omega(T_{\mathrm{I.11.(i)}}) $ given by \cite{Nur:OrbitsInvariantsCubicMatrices} and check that for all of them the Strassen flattening has rank at most $6$. We conclude
\[
\sigma_3(\bbP^2 \times \bbP^2 \times \bbP^2) = \scrF^1_{2,\bbP^2} \cap \scrF^2_{2,\bbP^2}.
\]
\end{proof}

One says that a tensor $T \in \bbC^n \otimes \bbC^n \otimes \bbC^n $ has maximal border subrank if $\sigma_n(\bbP^{n-1} \times \bbP^{n-1} \times \bbP^{n-1}) \subseteq \Omega(  T)$. The variety of tensors of maximal border subrank $\calQ_n$ is the closure of
\[
\calQ_n^\circ = \{ T \in \bbP(\bbC^n \otimes \bbC^n \otimes \bbC^n): \text{$T$ has maximal border subrank}\}.
\]
Border subrank of tensors is studied in additive combinatorics and quantum information theory \cite{ChrGesZui:GapSubrank,GesZui:NextGap}. From a geometric point of view, one can give an invariant theoretic description. We explain this for $n=3$, which is relevant to \autoref{prop: subrank components}. If $\dim V_i = 3$, the subring of $\bbC[V_1 \otimes V_2 \otimes V_3]$ of invariants for the action of $\SL(V_1) \times \SL(V_2) \times \SL(V_3)$ is a polynomial ring $R = \bbC[f_6,f_9,f_{12}]$, where $f_6,f_9,f_{12}$ are three invariants of degree $6,9,12$, respectively; see, e.g., \cite{BrHu:FundamentalInvariants333}. It turns out that the $f_6$ does not vanish on the unit tensor $T_{\mathrm{IV.ps}}$, whereas $f_9$ and a suitable linear combination $\tilde{f}_{12} = f_{12} + \alpha f_6^2$ do vanish on $T_{\mathrm{IV.ps}}$. The same proof as \cite[Prop. 5.1]{Chang:MaximalBorderSubrank} shows that 
\[
\calQ_3^\circ = \{ T : f_9(T) = \tilde{f}_{12} (T) = 0, f_6(T) \neq 0\}.
\]
As a consequence $\calQ_3$ is the union of the components of $ \{ T : f_9(T) = \tilde{f}_{12} (T) = 0\}$ not contained in the \emph{nullcone} for the action of $\SL(V_1) \times \SL(V_2) \times \SL(V_3)$ on $V_1 \otimes V_2 \otimes V_3$. We provide a characterization in terms of the collineation varieties.
\begin{theorem}\label{prop: subrank components}
The variety $\calQ_3$ has three components  
\[
\scrF^1_{2, S_{(1,2)}}, \qquad \scrF^2_{2, S_{(1,2)}} \qquad \text{and} \qquad \scrF^3_{2, S_{(1,2)}}.
\]
In particular, $ \scrF^i_{2, S_{(1,2)}} $ has codimension $2$ and degree $36$. 
\end{theorem}
\begin{proof}
By the results of \cite{Nur:OrbitsInvariantsCubicMatrices,DitDeGrMar:ClassificationThreeQutrit}, the set $\calQ_3^\circ$ coincides with the union of all orbits of tensors in \cite[Table IV]{DitDeGrMar:ClassificationThreeQutrit}: indeed, the unit tensor $T_{\mathrm{IV.ps}}$ is the semisimple part of such tensors, therefore it belongs to their orbit-closures. Using this, and the data of \autoref{table: semistable orbits}, we have 
\[
\calQ_3 \subseteq \scrF^1_{2, S_{(1,2)}} \cup \scrF^2_{2, S_{(1,2)}} \cup \scrF^3_{2, S_{(1,2)}}.
\]
In fact, for tensors $T_{\mathrm{IV.\alpha}}$ with $\alpha = 1 \vvirg 6$, this is immediate, because one of their collineation variety is indeed $S_{(1,2)}$. One can see that the tensor $T_{\mathrm{IV.7.(i)}}$ is contained in the orbit-closure of $T_{\mathrm{IV.5.(i)}}$. To see this consider $g_\eps \in \GL(V_1)\times \GL(V_2)\times \GL(V_3)$ which maps $g_\eps( b_0) = \eps^{-1} b_0, g_\eps (c_0) = \eps c_0$ and fixes all other basis elements. Then, $\lim_{\eps \to 0} (g_\eps \cdot  T_{\mathrm{IV.5.(i)}}) = T_{\mathrm{IV.7.(i)}}$, showing $T_{\mathrm{IV.7.(i)}} \in \Omega(T_{\mathrm{IV.5.(i)}}) \subseteq \scrF^1_{2, S_{(1,2)}} $.  All other elements satisfying $\scrC_2^i(T) = S_{(1,2)}$ appear in \autoref{table:unstable}, which classifies the unstable orbits following \cite{Nur:ClosuresNilpotent}. These are the elements of the nullcone of the action of $\SL(V_1) \times \SL(V_2) \times \SL(V_3)$ on $V_1 \otimes V_2 \otimes V_3$, that is the zero set of the three invariants $f_6,f_9,f_{12}$. We will show that the nullcone is contained in $\calQ_3$, which implies the other inclusion. 

Let $\calF = \{T : f_9(T) = \tilde{f}_{12}(T) = 0\}$. From the discussion above, $\calQ_3$ consists of the components of $\calF$ not contained in the nullcone. The results of \cite{Nur:OrbitsInvariantsCubicMatrices} guarantee that the nullcone has codimension $3$ in $\bbP ( V_1 \otimes V_2 \otimes V_3)$, whereas every component of $\calF$ has codimension (at most) $2$. Therefore, no component of $\calF$ is contained in the nullcone and $\calQ_3 =  \calF$. On the other hand, $\calF$ contains the whole nullcone, where all three invariants vanish. This completes the proof of the equality $\calQ_3 = \scrF^1_{2, S_{(1,2)}} \cup \scrF^2_{2, S_{(1,2)}} \cup \scrF^3_{2, S_{(1,2)}}$.

The statement on the degrees follows because the three components $\scrF^i_{2, S_{(1,2)}}$, for $i = 1,2,3$, are isomorphic: they can be obtained one from the other by permuting the tensor factors. Moreover, $f_9$ and $\tilde{f}_{12}$ have no common factors, therefore $\deg \calQ_3 = 9\cdot 12$ by B\'ezout's Theorem.  We conclude
\[
\deg \scrF^1_{2, S_{(1,2)}} = \frac{9 \cdot 12 }{3} = 36. \qedhere
\]
\end{proof}

We conclude this section observing that even in the small space $\bbC^3 \otimes \bbC^3 \otimes \bbC^3$, the collineation varieties define a finer invariant than tensor rank. It is known that tensors in $\bbC^3 \otimes \bbC^3 \otimes \bbC^3$ have rank and border rank bounded by $5$ \cite{BreHu:Kruskal333rank5}. Inside this tensor space, $\sigma_3(\bbP^2 \times \bbP^2 \times \bbP^2)$ parametrizes tensors of border rank at most $3$. Tensors of border rank at most $4$ form a hypersurface of degree $9$, cut out by the determinant of the Strassen flattening, that is the invariant $f_9$ mentioned before, see also \cite[Sec. 5.1]{ChrGesJen}. We exhibit two tensors of border rank $4$ and rank $5$ with different collineation varieties. Consider
\begin{align*}
    T_1 &=  a_0  b_0  c_1+a_1  b_0  c_0+a_1 b_1  c_1+a_2 b_1 c_2+a_2  b_2  c_0, \\
    T_2 &= a_0 b_0 c_1+a_0  b_1  c_2+a_1  b_0 c_0+a_2  b_1 c_0+a_2 b_2 c_1.
\end{align*}
One can verify $\rank(\Str(T_1)) = \rank(\Str(T_2)) = 8$, which guarantees that $T_1,T_2$ have border rank $4$. Moreover, one can show that $T_1,T_2$ have rank $5$: the upper bound is immediate whereas the lower bound is slightly involved and relies on the \emph{substitution method} of \cite{AlFoTs:TensorRankLowerUpperBounds}, see also \cite[Sec. 5.3]{Lan:GeometryComplThBook}. A direct check shows that $\scrC^1_2(T_1)$ is the quadric surface, whereas $\scrC^i_2(T_1)$ are rational normal scrolls for $i=2,3$. On the other hand, $\scrC^i(T_2)$ are rational normal scrolls for every $i=1,2,3$. This shows an example of two tensors of border rank $4$, rank $5$, such that the triples of collineation varieties are different. In particular, the values of rank and border rank do not determine the collineation varieties of the tensor.

\section{Future directions}

One natural open direction concerns the type of varieties that can arise as collineation varieties of tensors. By definition, every collineation variety is unirational and in the cases classified in this paper it is always rational and of minimal degree. It is clear that the minimal degree condition does not propagate to higher dimension; in fact, in general, collineation varieties are not even linearly normal, as observed in the following example:
\begin{example}
    Let $T \in V_1 \otimes V_2 \otimes V_3$ be generic with $\dim V_1 = 4$, $\dim V_2 = \dim V_3 = 3$. By genericity $B_{2}^1(T) = \emptyset$ because $\bbP T(V_1^*)$ is a generic $3$-dimensional linear space  and $\codim \sigma_1^{V_2 \otimes V_3} = 4$. Therefore, the linear system of quadrics parametrizing $\scrC^1_2(T)$ is base-point-free and its image is the projection of the Veronese threefold $\nu_2(\bbP V_1^*)$ from a point outside of it. In this case $\deg \scrC_2^1(T) = \deg \nu_2(\bbP^3) = 8$. 
\end{example}
We do not know whether collineation varieties are rational in general. However, it follows by the argument below that this is the case when the dimension of the linear spaces of matrices is preserved. In this direction, we expect general facts on the degree of projection maps and their image would shed some light on the type of collineation variety. In the case where the $k$-th base locus of a tensor is zero-dimensional, we can derive an upper bound on the degree of the collineation variety. Let $T$ be a tensor identified with $\PP(T(V_1^*))\cong \PP^a\subset \PP(\CC^{n_2}\otimes \CC^{n_3})$ such that $B^1_k(T)$ is zero-dimensional of degree $p$ and $\dim \scrC^1_k(T)=a$. Following \cite{BCD}, the algebraic multiplicity of $B^1_k(T)$ is defined as
\[
e(B^1_k(T)):=a!\lim_{n\rightarrow \infty}\dfrac{\dim_{\CC}\HH^0(\PP(T(V_1^*)),\cO_{\PP(T(V_1^*))}/\mathcal{I}_{B^1_k(T)}^n)}{n^a}, 
\]
where $\mathcal{I}_{B^1_k(T)}$ is the ideal sheaf of $B^1_k(T)$, and in general $\deg(B^1_k(T))\leq e(B^1_k(T))$ with equality if and only if the ideal of $B^1_k(T)$ is locally a complete intersection. By the degree formula \cite[Thm. 3.3]{BCD}, we have
\begin{equation}\label{eqn: bound degree}
\deg \scrC^1_{k}(T)=\dfrac{k^a-e(B^1_k(T))}{\deg \mu^{1}_{k,T}}.
\end{equation}
As the ideal of minors of a general matrix has a linear representation matrix, tensoring the corresponding exact sequence of $\C[x_{ij}]$-modules by $\C[x_{ij}]/I_T$ gives a linear representation matrix of the ideal of $B_k^1(T)$; here $I_T$ is the ideal of linear forms cutting  $\PP(T(V_1^*))\subset \PP(\CC^{n_2}\otimes \CC^{n_3})$. By~\cite[Thm. 3.2]{DHS}, the map $\mu^{1}_{k,T}$ is generically finite of degree $1$, hence birational. Therefore, $\deg( \mu^1_{k,T}) = 1$ in \eqref{eqn: bound degree} and the degree of $\scrC^1_{k}(T) $ is bounded above by $k^a-p$. This bound is achieved if $B_k^1(T)$ is a union of points in general positions.

One of the main directions left open concerns questions on the geometry and the basic invariants of the loci $\scrF^{i}_{k,X}$. 
Information on the number of irreducible components and the dimension of these varieties is important to understand whether they can be used to realize other interesting invariants of tensors. In the case of pencils, this amounts first of all to obtaining a generalization of \autoref{prop: small Fs for pencils}; this however seems non-trivial, because it is hard to keep track of the different ways a pencil can intersect the different loci of low-rank matrices, which is the condition which in turn gives rise to the irreducible components. For higher dimensions, the first step would be to obtain a full understanding of the loci $\scrF^1_{2,X} \subseteq \bbP (\bbC^3 \otimes \bbC^3 \otimes \bbC^3)$ for the five surfaces $X$ described in \autoref{sec: nets}.

As observed in \autoref{sec: stratification}, the loci ${\scrF^i_{k,X}}$ can be used to describe other interesting invariants of tensors. For instance, if $X = \bbP^{n-1}$ is a linear space, we have the inclusion
\begin{equation}\label{eqn: linear collineation}
\sigma_{n}(\bbP^{n-1} \times \bbP^{n-1} \times \bbP^{n-1}) \subseteq \scrF^1_{n-1,\bbP^{n-1}} \cap \scrF^2_{n-1,\bbP^{n-1}} \cap \scrF^3_{n-1,\bbP^{n-1}}.
\end{equation}
In particular, equations for $\scrF^i_{n-1,\bbP^{n-1}}$ give equations for the variety of tensors of minimal border rank, in the sense of \cite{JelLandPal:ConciseTensorsMinimalBR}. However, in general, the inclusion in \eqref{eqn: linear collineation} is strict. For instance, one can verify that when $n = m^2$, the matrix multiplication tensor $\mathbf{MaMu}_m \in \bbC^{m^2} \otimes \bbC^{m^2} \otimes \bbC^{m^2} $ has linear $(n-1)$-th collineation varieties, but it is not of minimal border rank. Understanding what is the geometric reason which makes the inclusion in \eqref{eqn: linear collineation} strict would provide new equations for the variety of tensors of minimal border rank. 

In the end, we know very little about the structure of the boundaries of the strata $\scrF^i_{k,X} \setminus {\scrF^i_{k,X}}^\circ$. For instance, as mentioned before, we do not know whether this boundary can be realized as a union of other strata. We expect this not to be the case as soon as there are infinitely many strata, but constructing an explicit example is challenging.

\bibliographystyle{alphaurl}
\bibliography{tensors.bib}

\end{document}